\documentclass[11pt]{amsart}
\usepackage[utf8]{inputenc}
\usepackage{graphicx,color}
\usepackage{color}
\newtheorem{theorem}{Theorem}[section]
\newtheorem{lemma}[theorem]{Lemma}
\newtheorem{proposition}[theorem]{Proposition}
\theoremstyle{definition}
\newtheorem{definition}[theorem]{Definition}
\newtheorem{example}[theorem]{Example}

\theoremstyle{remark}
\newtheorem{remark}[theorem]{Remark}
\usepackage[english]{babel}
\usepackage[T1]{fontenc}
\usepackage{amsmath}
\usepackage{ifthen}
\usepackage{mathrsfs}
\usepackage{psfrag}
\usepackage{graphicx}
\numberwithin{equation}{section}

\begin{document}
\title{PBNF-transform as a formulation of Propositional Calculus, I}
\author{Pelle Brooke Borgeke}
\address{Linn\oe us university}
\curraddr{}
\email{pelle.borgeke@lnu.se}
\subjclass[2010]{Primary}
\date {05 February 2026}
\begin{abstract} Here, in a series of articles, we show methods for calculating propositional statements using algebraic polynomials as \emph{symbols} for the connectives, which are named \emph{operators}. These polynomials originate from the transformation between the \emph{principles of duality} and the \emph{Disjunctive Boolean Normal Form}, DBNF, and they appear if we use a geometrization in the unit square and simple algebraic methods, modulo 2. This we call the PBNF-transform. PBNF stands for \emph{Polynomial Boolean Normal Form} as these families are based on DBNF involved here. In the first paper in this series, we show that statements can be mapped \emph{bijectively} into different polynomial families $g=g(p_j,q_k,a)$ belonging to $\mathcal{H}(g)$, which we call the \emph{The House} of PBNF. We can also replace the connectives of logic with PBNF, as the polynomials are, in fact, a \emph{geometrization} of these connectives; the systems are isomorphic. The benefit of this formulation of the Propositional Calculus(PC) is a near trivialization of the methods. No axioms are needed, no truth tables, just a list of polynomials (which in themselves are self-explanatory), the only law of inference is the rule of Substitution. 
\end{abstract}
\maketitle
\section{Introduction: Fundamentals}
We shall here use a \emph{Polynomial Boolean Normal Form}, short PBNF, for a mathematical formulation of the \emph{classical} Propositional Calculus (with 2 truth values, 1 or 0). Using standard mathematical techniques, we effectively address various problems in mathematical logic, such as proving statements and theorems, and even finding new ones. Our view is based on transform theory, so we use dual spaces for the PBNF-\emph{transform}, transforming from the DBNF(Disjunctive Boolean Normal form) formulation and the duals in the operator space. This gives us Boolean polynomials, algebra modulo 2, that use only one inference rule, namely the \emph{Rule of Substitution}. The typical axioms of logic are not needed, as the method uses a geometry of the unit square. Actually, if we instead refer to a theorem we prove here, we can see the substitution as a linearity process given in the theorem; in this way, the system has \emph{no logic rule of inference}. In fact, the PBNF-transforms could also be used directly as polynomial operators instead of the logical connectives, as PBNF is an isomorphic geometrization of the connectives.
The idea is to connect the \emph{principles of duality}, used by Church[3], and others, which are transformations between different operator spaces, here called $\mathcal{OP}$ (with index $_*$ for a normal space and $_{**}$ for a complement space) $$\mathcal{OP}_* \ni (\land, \lor, \Rightarrow, f) \mapsto (\lor,\land \not \Leftarrow, t) \in \mathcal{OP}_{**}.$$ If we see this transformation in Disjunctive Boolean Normal Form (we use $x$ and $y$ for $p$ and $q$ in other BNF), we get, if we take, for example, the conjunction above $(x\land y) \mapsto (x\land y) \lor (x\land y') \lor (x'\land y)$. The next step is to transform the DBNF:s to the polynomial families $g(p_i,q_j,a_0)\in \mathcal{H}(g)$, which we call \emph{the House} of PBNF. We can then write \begin{align}p \land q \mapsto pq \in g(p,q,1) \mapsto pq+p(q+1)+(p+1)q \\ =(p+1)(q+1)+1 \in g(p',q',1'),\end{align} which now is on PBNF. This \emph{dual-normal form-polynomial} transformation\footnote{Some attempts in this direction have been made in examples by Stoll [8] and Halmos [4], and in Bergmann [2], we find another arithmetic method with min and max calculations.} which we call the PBNF-transform, gives a Boolean algebra, modulo 2, turning propositional calculus into algebraic calculations. A mapping $f$ takes the arguments in $\mathcal{OP}$, which here consists of classes of operators, null-class $(p,q$, operands or letters), singular and binary operators, single statements and compound statements, which we denote by $S_{c(k,l)},$ and give us linear and quadratic polynomial forms, so that we could, at once, calculate, in a trivial fashion, the statement with the algebraic methods in $\mathbb{Z}_2$.
We can change, to fit the statement, to a different polynomial family in the \emph{House}. This can be done, for example, by choosing the complement of a family, renaming truth values, or using different inputs for the operators, which we write as vectors, $op_{\boldsymbol{x}(x_1,x_2,x_3,x_4)}$ or just $op_{\boldsymbol{x}}$. The input $(p_i,q_j)$ is decided by the \emph{Selectors}, $\mathcal{S}$, to the (binary) polynomial families in $\mathcal{H}(g).$ In the \emph{protocol} (rules and conditions) below, we outline how it is done. ($(\cdot)$:=for all, [$\cdot]$:=exists, |:=XOR, $\boldsymbol{.}:=(\cdot),$ inner parenthesis.) \begin{align}(g)[\mathcal{H}][f][a_j][p_j,q_k]  \ \bigr(g \in \mathcal{H}, f:\mathcal{OP}\to (\mathbb{Z}_2 | \mathcal{H}(\mathbb{Z}_2)), f(p .op_{\boldsymbol{x}} \ q)\\ = g(p_{j},q_{k}, a_0) = a_1pq+a_2p+a_3q+a_4 =(1|0, \big{|}g(p_i,q_j,a_0)(\mathbb{Z}_2)\big{|} \ \hookrightarrow \mathcal{OP}), \\ a_{0\leq n \leq 4}=0|1, i=1,2,3,4 \land j=1,2. \end{align}
We observe that the range of $f$ is real valued both in $\mathbb{Z}_2$ and in $g \in \mathcal{H}$. The polynomial $(a_1pq+a_2p+a_3q+a_4)(1,0)$ can be evaluated for $\mathbb{Z}_2$ if we do not reach a tautology (1) or contradiction(0). It can also be used in the inverse action(often here called the \emph{fiber} or the \emph{pull back}) $g^{-1}(a_1pq+a_2p+a_3q+a_4) \hookrightarrow \mathcal{OP}$.
For example, one may wonder if the following statements are equal $$(p \land q) \lor r = p \land (q \lor r).$$ 
The transformations, for both sides of the equality are in the family $g(p,q,1)$: $$\mapsto (p^S+1)(q^S+1)+1= p^Sq^S .$$
The polynomials after substitution \begin{equation} (pq+1)(r+1)+1 = p(q+1 .r+1 +1).\end{equation} The statements are not equal, because they give different polynomials in the same family. If we solve (1.6) and simplify, we get $r=pr$ , and for the spin of, the true statement $$(p \land q) \lor (p\land r) = p \land (q \lor .p\land r).$$ We can now substitute back to the letters $p$ and $q$ by putting $r=q$ to get \footnote{With $p\lor q \mapsto (p^S+1)(q^S+1)+1$, $p\land q \mapsto pq^S $ and $S$ for substitution, we get LS:$(pq+1)(pq+1)+1=pq,$ RS$:p(q+1 .pq+1+1)=p(pq+q+pq+1)+1)=pq.$} $$(p \land q)\lor(p\land q)=p \land (q\lor.p\land q)$$ which also is true. PBNF is a canonical form, which means that two logically equivalent formulas convert to the same PBNF, easily showing whether two formulas are identical for automated theorem proving. This is not the case, for example, in DBNF, $p \lor \lnot p$ is in DBNF but is not reduced to 1, although they are logically equivalent. In PBNF, we get \footnote{The polynomial $(p+1)p+1$ does not look like a disjunction, it is $(\lnot p \land p)'= p \lor \lnot p$ by De Morgan and used to get 1=T. If we take 0=T, we get $p(p+1)=0$ instead because in the space $g(p,q,0)$ we have $\land \mapsto \lor.$} when calculating the polynomial, with addition modulo 2 and $p^2=p$,$$(p \lor \lnot p) \mapsto (p+1)p+1=p^2+p+1=p+p+1=1\in g(p,q,1).$$  Or we see that $(p+1)p=0$ as the dual statement $(p \land \lnot p)$.
It is also possible to put the operator $\textbf {x}(x_1,x_2,x_3,x_4)$, which we think of as a vector, on Boolean matrix-form, $\left[ {x_1 x_2 \atop x_3 x_4} \right]$ to get a new view and, if we do so we find that we can take just $p$ or $q$ as a minimal primitive connective forming an adequate system (just $\uparrow$ and $\downarrow$ are considered as adequate mono-connectives in the vector case). We can also shift from the vector case to a system of matrices in our transformation, where $N(op_{\boldsymbol{x}})$ takes the numberstring from an operator,
$$N(\downarrow )=0001 \mapsto \left[ {00 \atop 01} \right] \mapsto \left[ {p+1 \ 0 \atop 0 \ q+1 } \right] \in g(p,q,1) = (p+1)(q+1).$$  
Observe that det $\left[ {0 0 \atop 0 1} \right]=0$, which means that the matrix is Boolean, which applies in all non-trivial cases(N$(\Leftrightarrow)=\left[ {10 \atop 01} \right]$ is Booelan but determinant is 1 for this unit matrix. Observe also that $\left[ {(p+1) \ 0 \atop   0 \ (q+1) } \right]^2=\left[ {(p+1) \ 0 \atop 0 \ (q+1) } \right]$ so this formulation stays Boolean.
A brief overview of the material: In Section 2, we start by looking at the operators, which are the arguments of our PBNF-transform, and we present more examples using the polynomial families. The proofs of the theorems come in Section 3, where we also go deeper into the material. 
The more advanced applications start, for example the generalization of the dual theorems found in Church [3].
\section {Operators as arguments for the PBNF-transform}
We shall now present our formulation of Propositional Calculus in some more details. This section can be seen as a more thorough introduction than the first. In the next chapter, we go deeper into the subject and prove the theorems. 
We start by discussing the different types of operators we start from. These operators are the arguments of the PBNF-transform that gives us the polynomials, so that for example $g_{(p_i,q_j,1)}(p\Leftrightarrow q)=p+q+1 \in g_{(p_1,q_1,1)}$. We begin with the \emph{Selectors}, $\mathcal{S}$, which we index here as $p_i, q_j, i=1,2,3,4, j=1,2$. Note that $q$ can have two faces $(1010-q_1, 0101-q_2)$ but $p$ four $(1100-p_1, 0110-p_2, 0011-p_3, 1001-p_4)$. (We use this order, although $(p_1)'=p_3,$ and $(p_2)'=p_4.$) We have here a formula when $p$ and $q$ can be used and not (recall that we must not have 0 or 1 when we add the numberstring, 1100+0011=1 or 1100+1100=0, but 1100 1010=0110 works). We get \begin{equation}\left\{  
\begin{array}{ll}
p_i+ p_j=q_{i+j}\\
p_i+ q_j=p_{i+j}\\
\end{array}\right.
\end{equation}
This shows that we must not use two of the same, or $p|q$ and the pullback of $p|q$. The Selectors must give two homogeneous and two heterogeneous combinations of the elements $1$ and $0$. We have $p_1(1,1,0,0)$ and $q_1(1,0,1,0)$ which we call the $1^{st}$ and $2^{nd} Selector$ (usually just $p$ and $q$) with different truth values (here $1=$ True and $0$=False, note we also use the opposite by renaming $0=$ True and $1$=False) for example in the maps \begin{equation} \frac {\Rightarrow}{\Leftrightarrow} (\begin{smallmatrix}p_1 & 1 & 1 & 0 & 0\\ q_1 & 1 & 0 & 1 & 0 \\ \end{smallmatrix}) = (\begin{smallmatrix}\Rightarrow  & 1 & 0 & 1 & 1\\ \Leftrightarrow & 1 & 0 & 0 & 1\\ \end{smallmatrix}) \quad \frac {\Rightarrow}{\Leftrightarrow} (\begin{smallmatrix}p_3 & 0 & 0 & 1 & 1\\ q_{2} & 0 & 1 & 0 & 1 \\ \end{smallmatrix})= (\begin{smallmatrix}\Leftarrow  & 1 & 1 & 0 & 1\\ \Leftrightarrow & 1 & 0 & 0 & 1\\ \end{smallmatrix}).\end{equation}
In the next shift we find that $\Leftrightarrow$ turn into $p$ and $p'$ which give us the idea that these selectors are also binary operators (we shall explore that later).  \begin{equation} \frac {\Rightarrow}{\Leftrightarrow} (\begin{smallmatrix}p_ 4& 1 & 0 & 0 & 1\\ q_1 & 1 & 0 & 1 & 0 \\ \end{smallmatrix})= (\begin{smallmatrix} \Leftarrow  & 1 & 1 & 1 & 0\\ p & 1 & 1 & 0 & 0\\  \end{smallmatrix}) \quad \frac {\Rightarrow}{\Leftrightarrow} (\begin{smallmatrix}p_4 & 1 & 0 & 0 & 1\\ q_2 & 0 & 1 & 0 & 1 \\ \end{smallmatrix})= (\begin{smallmatrix} \uparrow  & 0 & 1 & 1 & 1\\ p_3 & 0 & 0 & 1 & 1\\ \end{smallmatrix})\end{equation}
This technique helps us choose the most convenient polynomials for a specific statement.
We write a \emph{truth-valued vector operator} as $op_{\boldsymbol{x}(p_n,q_n)(x_1,x_2,x_3,x_4)}$, in the case $\Rightarrow $ we get $op_{\boldsymbol{x}1,0,1,1}$ where $x(1,0,1,1)$, is the range for the different inputs of the selectors, and we omit the selectors which use to be clear from the context. A value in the range for the operator, or an output, we often refer to as the $\emph{vote}$ or the $\emph{pick}$ of the operator. 
The geometrization of the algebra of logic takes place in a Venn diagram of the unit square (for more of this, and lattices too, consult [10]). Two selectors represent a partition of the unity in 4 different areas (1, 2, 3, and 4) of two sets, we take $p_1(1100)$ and $q_2(1010)$, (we call $x_1x_2x_3x_4, x=0| 1$ for a number string of the operator, which is given on vector form), with a non-void intersection in $p\cap q)$. This is area number 1. Number 2 is $p \setminus q$, number 3 is $q \setminus p$, and the complement $(p' \cap q')$ is number 4. If we compare $p(1100)$ and $q(1010)$, we see that each number represents an area, with $p$ to the left of $q$, so it is clear which areas belong to $p$, to $q$, or to both. Every binary operator stands for specific areas; for example, the disjunction $\lor$ has the number string (1110), so it occupies every area but 4. 
We can thus change to different operator spaces when we change inputs. In all, a binary operator space $\mathcal{OP}$ contains $16$ different binary operators with a four-place string $(1|0,1|0,1|0,1|0)$ output. In the case when $p=q$, we have only 2 areas in the unit square: $p$ and $p'$. We then let $p$ be area 1 and $p'$ be area 2. This geometry is used for singular operators, which we have $4(1|0,1|0)$ of, some of which are not named or considered in most of the literature, more than the negation, which we here often refer to as the \emph{pull back action}; however, they can be used in different situations in the Calculus of Statements we work with. The protocol for singular operators is ($[\cdot]$:=exists, |:=XOR)  \begin{align}[f][\mathcal{H}][a_j][\text{T|F}][p_j|q_k] \bigr(f:\mathcal{OP}\to (\mathbb{Z}_2 | \mathcal{H}),f(op_{\boldsymbol{x}(x_1,x_2}):= f(p_j|q_k,a_0)(op_{\boldsymbol{x}})\\= a_1p|q+a_0 =(1|0, \big{|}f(a_1p_j|q_k,a_0)),  a_{n=0,1}=0|1,\\ (g(p_j|q_k,a_0))'=g(p_j'|q_k',a_0'), g(p_j|q_k,\underline{a_0}) \mapsto (1|0=\text{(T|F)|(F|T)}), \\ p_j|q_k\in \mathcal{S}=(1|0) \bigl). \end{align} Here we can use either $p$ or $q$ because the polynomial family is linear in $p$ or $q:(0,p|q,p|q+1,1).$ We use $f$ instead of $g$ for these linear polynomial families.
Propositional or Statement calculus here means that we work with $\mathcal{H}(f,g)$ instead of $op_{\boldsymbol{x}(x_1,x_2,x_3,x_4)}(p_n,q_n) \in \mathcal{OP}.$ Every operator is bijectively mapped into $\mathcal{H}(f,g)$, the \emph{House} of linear and quadratic polynomial in $p$ and $q$, so for example for the bi conditional\footnote{It is not allways the case that $\lor \mapsto +$, for this to happen we need disjoint sets, $x\cap y= \emptyset.$ If not, $x \lor y \mapsto p+q+pq=(p+1)(q+1)+1 .$}, and here we show the middlestep in DBNF, \begin{align}op_{\boldsymbol{x}(1,0,0,1)}(p_1,q_1) \mapsto (x\land y) \lor (x'\land y') \mapsto pq+(p+1)(q+1) \\ =p+q+1  \in g(p,q,1)\subset\mathcal{H}(f,g). \end{align} But in another family, we get \begin{equation}op_{\boldsymbol{x}(1,0,0,1)}(p_1,q_1) \mapsto p+q \in g(p',q',0)\subset\mathcal{H}(f,g). \end{equation} as this is the complement of the first case. To calculate a statement means to prove or disprove it, then it is called a tautology (Taut=1/0), true for every combination of input, or a contradiction (Taut=0/1)$'$. We want to know if have a tautologi in $(p \dot \land(.p\Rightarrow q) \dot\Rightarrow q$, so we map this into \footnote{Here $p\Rightarrow q=\lnot p \lor q=(p\land \lnot q)'\mapsto p(q+1)+1 \in g(p,q,0)$, and if in $g(p',q',0)$ we have $(p+1)q$.} \begin{align} (p \dot \land .p\Rightarrow q) \dot\Rightarrow q \mapsto p^S(q+1)+1 \in g(p,q,1)\\ =p(p . q+1+1)(q+1)+1p=(p.q+1+p)(q+1)+1\\=pq(q+1)+1=1.\end{align} The mapping from the principal operator, here marked with a dot for $(\Rightarrow)$, is a conditional, and S means we substitute for $p$. There we have a conjunction, which is principal: $p\land q \mapsto pq^S$. Finally, we substitute for the conditional $q^S=(p(q+1)+1)$, and  we get (1.6) with the help of Boolean arithmetic where $x+x=0$ and $x \cdot x=x.$
With PBNF, we can also, as we saw before, get a bonus for achieving results when the calculation yields values other than 1/0, e.g., an elimination of connectives (we use $\to$ | $\Rightarrow$)  $$ (.A\land B \to C)\Leftrightarrow (A\to .B\to C).$$
Let's say that we want to eliminate the conjunction to the left to prove the statement. $(A(p)\land B(r))\to C(q)) \mapsto p^S(q+1)+1=pr(q+1)+1.$
The right side $A(p)\to (B(r)\to C(q)) \mapsto p(q^S+1)+1=p(r.q+1+1+1)+1=pr(q+1)+1.$
Of course, we can prove the whole statement true. Then we shall ad LS+RS+1, as $p\Leftrightarrow q \mapsto p+q+1 \in g(p,q,1)$. Thus we get $$((A\land B)\to C)\Leftrightarrow (A\to (B\to C)) \mapsto p^S+q^S+1$$ $$=(pr(q+1)+1)+(pr(q+1)+1)+1=1$$ and this means that we have a tautologi. We could also rewrite $$pr(q+1)+1=p(r(q+1))+1=p(r(q+1)+1)+1)+1$$ to eliminate $pr \hookrightarrow p \land r$. \footnote{Here, perhaps $g(p',q',0)$ would be simpler to use, $p \to q\mapsto (p+1)q$, but we also have $\land \mapsto (p+1)(q+1)+1 \in g(p',q',0).$  We often call $g^{-1}$ for the \emph{fiber} and we use $\hookrightarrow.$} \begin{definition} The connectives or the operators belong to the class of operators which we name $\mathcal{OP}$. We have 4 singular operators that are connected to one statement letter, one, and the most used of them is the negation $(\neg p)$, which we number as $1-4).$ They belong to $\mathcal{OP}_1(op_{\boldsymbol{x}(x_1,x_2)})$. There are 16 binary operators (numbered 5-20) that connect statement letters (literals, operands), e.g. $p\land q$, which is one of the most common $(\vee, \wedge, \Rightarrow, \Leftrightarrow) \in \mathcal{OP}$. They belong to $\mathcal{OP}_2(op_{\boldsymbol{x}(x_1,x_2,x_3,x_4)}).$
We divide these binary connectives into the primitive operators, $\mathcal{OP}_{P}$ needed to start with, where different choices can be made (for example, $\lnot,\land$ or $\lnot,\lor$). \footnote{These primitive connectives are often argued to be as few as possible to bring clarity to the Propositional Calculus (PC), but it is also noted that this reduction brings problems as the statements get more complicated. One of these \emph{1-primitive formulations} uses the \emph{Sheffer stroke}, often a vertical line |, but here we use $\uparrow$. We shall later give examples of this, showing that Selector$(p|q)$ can be used as a 1-primitive if we may use matrix algebra with addition and multiplication. The “primitive polynomial” in PBNF can be thought of as $(p, q, 1)$, two Selectors, and 1 for binary operators and $(p, 1)$ or $(q, 1)$ for singular operators. The idea of 1-primitive has, in fact, no importance in the polynomial formulation, as there is just this choice to be made. Instead we see $(p, q, 1)$ as \emph{generators} of the full polynomial, $pq+p+q+1=(p+1)(q+1) \hookrightarrow p\uparrow q $ (nondisjunction, $\lnot \lor$). The other 1-primitive we have is $pq+1 \hookrightarrow p\downarrow q$ (non-conjunction, $\lnot \land$), which shows examples of the polynomial as an algebraic-geometric formulation of PC.} Then there are normal form operators $ \mathcal{OP}_{NF} (\vee, \wedge, \neg$) and the group we call derived operators $\mathcal{OP}_{D}(\Rightarrow, \not \Rightarrow, \Leftarrow, \not\Leftarrow, \uparrow, \downarrow $). These are all unevenly heterogeneous, so that we could get an even operator by adding number strings.
Next group is the $\emph{Selectors}$, these are all the operators with two 1 and two 0 (equally heterogeneous), so that they could act as inputs $\mathcal{OP}_{S_1}(p, q, p', q')$ and $\mathcal{OP}_{S_2}(\Leftrightarrow, | $) or just $\mathcal{OP}_{S}=(p_1, p_2, p_3, p_4; q_1, q_2).$  The selectors $\mathcal{OP}_{S_2}$ are usually treated only as binary operators. For $\mathcal{OP}_{S_1}$, it is the opposite; these operators are usually just inputs, but we show here that they also could be seen as binary operators, with special qualities.
The fourth group consists of the Trivial operators \footnote{In most formulations of PC the Trivial operators are just there, but they are not that trivial, as they are needed if you use 1 | 0 as truthvalues, and even if you use T/F (or t/f) these letters just appear to be there. In Church´s formulation $P_1$, the primitives are $[\Rightarrow]$ (the brackets and the conditional) and the primitive constant $f$. He starts by defining $(\to $ := stands for), $1 \to f \Rightarrow f$ and then $(\lnot p \to p \Rightarrow 0 $, which is enough, and quite slick. The trivial operators can be put between letters to form a lattice of \emph{logical operators}, which make an evaluation over the inputs depending on their truthvalues. $T_0$ does not believe in any input, but $T_1$ belives in every input. Other operators in this group is $\land$ and $\lor$, but also the Selectors, $p_1, q_1$. More of this in coming papers.} $\mathcal{OP}_{T_1}$ and $\mathcal{OP}_{T_0}$ (1 and 0), which are the only homogeneous binary operators. Now we can form a linear order because the first operator can build the next one, and so on. \begin{align} \mathcal{OP}_{P}-\mathcal{OP}_{NF}( \land, \lor \lnot) - \mathcal{OP}_{S}(p,q,p',q',\Leftrightarrow, | ) \\-\mathcal{OP}_{D}(\Rightarrow, \not\Rightarrow,\Leftarrow, \not\Leftarrow, \downarrow, \uparrow ) - \mathcal{OP}_{T_1}\subset \mathcal{OP}. \end{align} This is not a strict order; for example, you could construct the normal form operators by using $p\Rightarrow q = \lnot p \lor q$ and  $p \not\Rightarrow q =  p \land \lnot q.$  Observe that this order is outside of the normal form order, which is ordered by the numbers of 1 or 0 in $op_{\boldsymbol{x}}$ $$\longrightarrow \text{Increasing order of ones}$$ $$ op_{\boldsymbol{0}}(\iota_0) \subset op_{\boldsymbol{1}}(\land, \not\Rightarrow, \not\Leftarrow, \downarrow)\subset op_{\boldsymbol{2}}(p,q,p',q',\Leftrightarrow, |)$$  $$op_{\boldsymbol{3}}(\lor \Rightarrow, \Leftarrow, \uparrow)\subset op_{\boldsymbol{4}}(\iota_1)$$ $$\longleftarrow \text{Increasing order of zeroes}$$ This will be shown later when we present the lattices in the different cases, and, as noted before, it is possible to shift $0\mapsto 1$ and $1\mapsto 0$ for the truth and false values.
The operators are used to create formulas or statement forms and result in the $\emph{operator space}$. They could be connected to the duals in the operatorspaces $\mathcal{OP}_x$ or directly to $\mathcal{H}(f,g)$, which is the family of polynomials associated with the operator space, $\mathcal{OP}_x$. The quantifiers $\forall x$ or here $(x)$ and $\exists$ or here $[x]$ are usually not called operators but $\emph{constants}$ in $\mathcal{OP}$ together with the \emph{primitive} “=” and $\in$. Here, we also refer to them as null-class operators. While operators are put between statements to form other statements, “$\in$” and “=” are put between constants to form statements. For example we form from the constants in $\mathcal{OP}$, $p,q,r,s,p_1,q_1,r_1,s_1 \ldots$, the statement $p=q$ or $p\in q$. We also put $p,p'$ and $q,q'$ in $\mathcal{H}(f,g)$, which have a double nature of being both constants, statement letters, (and we could use p and q if necessary) and operators (The Selectors $p, q, p',q'$). When we have this group, the fifth in $\mathcal{OP}$, we get a total of 24 operators: 4 are singular, 10 are ordinary binary, 2 are trivial, 4 are bi-constants (both selectors/operators and statement letters), and 4 are constants. We also refer to the statement letters, the \emph{operands}, as null-class operators, group 5', they start at 25 and up. 
\end{definition} \begin{table} \begin{tabular}{c r l}
\emph{The operator space} $\mathcal{OP}$ & $Number$& $ Names$\\
\hline
$(1)$ Singular operators  & $1-4$ & $=, -,\lnot, + $\\
$(2)$ Binary operators & $5-14$ &$\land,\lor, \Rightarrow \Leftarrow , \not \Rightarrow, \not \Leftarrow, \Leftrightarrow, |, \downarrow, \uparrow    $\\
$(3)$ Trivial operators & $15-16$ & $Zero(\iota_1), One (\iota_0)$ \\
$(4)$ Bi-constant operators &$17-20$ &$p, q, p',q'$ \\
$(5)$ Constants(null-class operators)  & $21-24$ & $\in,=, (\cdot)$:= for all, $[\cdot]$:= exists\\
$(5')$ Operands (null-class operators)  & $25-$ & $p,q,r,s,p_1,q_1,r_1,s_1 \ldots $\\
\end{tabular} \end{table}  \begin{remark}It can be noticed for the conditional $p\Rightarrow q$ that this operator belongs to a family of $\emph{Conjugate Sentences},$ by Tarski [9], along with the converse $p\Leftarrow q,$ the inverse $\lnot p\Rightarrow \lnot q$ and the contrapositive $\lnot p\Leftarrow \lnot q$. This will be shown when we later prove $\mathcal{H}^{*}$ where these conjugate sentences are transformed to polynomials in $p$ and $q$. We can already see what will happen, and the polynomials on the far right show the connections.
\end{remark} \begin{scriptsize} \begin{tabular}{c c c c}
\emph{Conjugate Sentences} & $Name$& $ Polynom \ g(p,q,1)$& $Polynom \ g(p',q',0)$ \\
\hline
$(1) p\Rightarrow q $ & $Conditional$&$p(q+1)+1$ & $(p+1)q$ \\
$(2) \lnot p\Rightarrow \lnot q$ & $Inverse$ &$(p+1)q+1$&$p(q+1)$\\
$(3) p\Leftarrow q $& $Converse$&$(p+1)q+1$& $p(q+1)$ \\
$(4) \lnot p\Leftarrow \lnot q$&$Contrapositive$ & $p(q+1)+1$ &$(p+1)q$.\\
\end{tabular} \end{scriptsize} 
We observe that $1=4$ and $2=3$ (to be proved in the next section), which gives us the opportunity to choose for the proof of the equivalence $$p\Leftrightarrow q= (p\Rightarrow q \land p\Leftarrow q), \text{Conditional-Converse},$$ the following $$p\Leftrightarrow q = (p\Rightarrow q \land \lnot p \Rightarrow \lnot q), \text{Conditional-Inverse},$$ $$p\Leftrightarrow q= (\lnot p\Leftarrow \lnot q \land p\Leftarrow q), \text{Contrapositive-Converse},$$
$$p\Leftrightarrow q= (\lnot p\Leftarrow \lnot q \land \lnot p \Rightarrow \lnot q), \text{Contrapositive-Inverse}$$.
\section{The polynomial families in the \emph{House} of PBNF}
What we have so far will give some examples of proof and definitions technique, where the statement calculus is working out. Later, we will go deeper into the material, and prove the theorems we use.  
Here, we study and use the four most natural families with linear polynomial $f(p,a$) \begin{align}1. f(p,1),  \text{Little normal family} : \text{normal input and reading}. \\ 2. f(p,1)\mapsto f(p,1'), \text{Little complement family} :  \text{reading 1(T) to 0(T)}. \\ 3. f(p,1)\mapsto f(p',1), \text{Little pullback family} : \text{input $p$ to $p'=p+1$}. \\ 4. f(p',1)\mapsto f(p',0), \text{Little pullback complement family} : \\ \text{change input $p$ and reading}.\end{align}
We will comback to explain the rules and the conditions in this family, that not coinside (recall that we have a different geometry here, just $x=p$ and $\lnot x= p+1$) with the quadratic polynomials $\in \mathcal{H}(g)$. We call \begin{align}1. \mathcal{H}^{*} g(p,q,1), \ \text{Normal Family},\\ 2. \mathcal{H}': g(p,q,1'=0), \ \text{Complement Family},\\  3. \mathcal{H}^{''}: g(p',q',1),\ \text{Pullback Family},\\  4. \mathcal{H}^{**}: g(p',q', 1'), \ \text{Pullback Complement Family} \end{align} \begin{remark} Note that $'$ means adding 1=1111, to the polynomial binary forms, negating or pull back a statement. We collect the polynomial families into two groups, depending on their truth-value order: $(1,0)$, \emph{normal families}, or $(0,1)$, \emph{complement families}. Of course, there are hybrid families, for example, $g(p,q',0)$ or $g(p',q, 1)$, which can be convinient to use sometimes. We are not going to investigate all the polynomial families as we can see the pattern just using the first four, the other ones are just shifting places of $op_{\boldsymbol{x}(x_1,x_2,x_3,x_4)}$ or changing values for truth and false. To actually count the number of families, we start with the first selector $p$, which has four permutations $(1,1,0,0),(0,1,1,0)$, $(0,0,1,1)$, and $ (1,0,0,1)$, the Second Selector $q$ has two $(1,0,1,0) (0,1,0,1)$. If we allow all six possibilities  $(p,q,p',q',p^{|},p^{\Leftrightarrow })$ as selectors, we get $6\cdot 4=24$ possibilities of dual-spaces  (recall that the complement($\textbf{x+x$'$}=1$) and the operator ($\textbf{x+x}=0$), or two of the same do not work together). These families can be read from bottom (normal space) or top(complement space) in their lattices meaning that if we use $(1,0)$ or $(0,1)$ as showing True and False we get $48$ dual spaces or that we have 47 dual spaces to the base-space that use $1^{st}$ and $2^{nd} Selector$ as inputs. \end{remark} \begin{definition} The polynomial family has four $\emph{particles}$ or simple forms made up with the variables $p$ and $q$ that correpond to the DBNF$(x(p),y(q))$ $pq=x\land y, p(q+1)=x\land y', (p+1)q=x'\land y$ and $(p+1)(q+1)=x'\land y'q$. These DBNF:s are usually named $\emph{phrases}$ and in CBNF $\emph{clauses}$ and the transformation is by De Morgan negating so e.g. $\lnot(x'\lor y')=x \land y $. In the polynomial families we instead add 1, $(p+1 .q+1 +1)'=(p+1+1)(q+1+1)+1+1=pq \ (\text{modulo} \ 2).$ \end{definition} Here is a table of the most common polynomial families described above. In the next Section, we prove this schema.
{\tiny \begin{center}
\begin{tabular}{||c | c |c |c |c ||}
\hline
$OP$ & $\mathcal{H}^{*}: g(p,q,1)$&$\mathcal{H}': (g(p,q,1'))$ & $\mathcal{H}''(p',q',1)) $ & $\mathcal{H}^{**}: (g(p',q',1')$ \\ [0.5ex]
\hline\hline
$ p $ & $p(1,1,0,0)$ &$p'$ & $p'$& $p$\\ 
\hline
$ p' $ & $p'(0,0,1,1)$ &$p$ & $p$& $p'$\\
\hline
$ q $ & $q(1,0,1,0)$ &$q'$ & $q'$& $q$\\
\hline
$ q' $ & $q'(0,1,0,1)$ &$ q$ & $q$& $q'$\\
\hline
$\neg p $ & $p+1$ &$p$ & $p$& $p+1$\\
\hline
$p\vee q $& $(p+1)(q+1)+1$ &$(p+1)(q+1)$ & $pq+1$& $pq$ \\
\hline
$p\wedge q$ & $pq$ & $pq+1$ & $(p+1)(q+1)$& $(p+1)(q+1)+1$\\
\hline
\hline
$p\Rightarrow q$  & $p(q+1)+1$ & $p(q+1)$ & $(p+1)q+1$ & $(p+1)q$ \\ 
\hline
$p \not\Rightarrow q$  & $p(q+1)$ & $p(q+1)+1$ & $(p+1)q$ & $(p+1)q+1$ \\
\hline
$\lnot p\Rightarrow \lnot q$  & $(p+1)q+1$ & $(p+1)q$ & $p(q+1)+1$ & $p(q+1)$ \\
\hline
$p \Leftarrow q$  & $(p+1)q+1$ & $(p+1)q$ & $p(q+1)+1$ & $p(q+1)$ \\
\hline
$p \not\Leftarrow q$  & $(p+1)q$ & $(p+1)q+1$ & $p(q+1)$ & $p(q+1)+1$ \\
\hline
$ \lnot p \Leftarrow \lnot q$  & $p(q+1)+1$ & $p(q+1)$ & $(p+1)q+1$ & $(p+1)q$ \\
\hline
\hline 
$p\downarrow q$ & $(p+1)(q+1)$ & $(p+1)(q+1)+1$ & $pq$& $pq+1$\\
\hline
$p\uparrow q$  & $pq+1$ & $pq$ & $(p+1)(q+1)+1$ & $(p+1)(q+1)$ \\
\hline
$p\Leftrightarrow q$ & $(p+1)+(q+1)+1 $ & $(p+1)+(q+1) $ & $(p+1)+(q+1)+1$ & $(p+1)+(q+1) $ \\
\hline
$p | q$ & $(p+1)+(q+1) $ & $(p+1)+(q+1)+1 $ & $(p+1)+(q+1)$ & $(p+1)+(q+1)+1 $ \\
\hline
$\iota_1 $ & $(p+1)pq+1 $ & $(p+1)pq$ & $p(p+1)(q+1)$ & $p(p+1)(q+1)+1 $ \\
\hline
$\iota_0$ & $(p+1)pq $ & $(p+1)pq+1$ & $p(p+1)(q+1)+1$ & $p(p+1)(q+1) $ \\ [1ex]
\hline
\end{tabular}\end{center} }
But first, some definitions and examples to clarify things. \begin{definition} By a statement, we mean a declarative sentence that can be classified as true (T or 1/0) or false (F or 0/1), and this is called the sentence's truth value. For every statement, we can name them $p,q,r,s,p_1,q_1,r_1,s_1, \ldots$ and the single statements we name \emph{letters}, \emph{literals}, or \emph{null-class} operators. The literal can be combined with connectives or operators to form compound statements. \end{definition} \begin{definition} The statements which we group in \emph{statement forms}, $S_{c(k,l)},$ are of different classes depending on how many different letters and how many connectives they have. In the class $S_{c(k\ge 0, \ l\ge 1)},$ where $k$ is counting the number of connectives and $l$ is counting the number of different literals. With $c$, we indicate that the Statement Calculus is classical, meaning we have two truth values: 1 and 0. In some non-classical systems, we can have several truth values, and the statement has a dependence on distance($x$) or time($t$), and we then use $S_{n(k,l)}(x,t)$ and $n \in \mathbb{N}$ is then the number of truth values. We also use other capital letters, such as $T_{c(k,l)}$ or $R_{c(k,l)}$ to indicate different forms. \end{definition} \begin{definition} A statement form $S_{c(k,l)}$ \emph{logically implies} or \emph{entails} another statement form $T_{c(k,l)}$ iff every assignment of (1,0) making $S$ true also makes $T$ true. We  write $S_{c(k,l)} \vdash T_{c(k,l)}$ for this fact.
$S_{c(k,l)}$ is a \emph{tautologi} if it only takes the value 1/0(true) for every input of truth values. $S_{c(k,l)}$ is a \emph{contradiction} if it only takes the value 1/0(false) for every input of truth values. \end{definition} \begin{theorem} $$(S_{c(k,l)} \vdash T_{c(k,l)}) \Leftrightarrow (S_{c(k,l)} \Rightarrow  T_{c(k,l)}) \ \text{is a tautologi}.$$ \end{theorem} \begin{proof} We can use the following from the introduction \begin{equation} (S_{c(k,l)} \vdash T_{c(k,l)})  (\begin{smallmatrix}p & 1 & 1 & 0 & 0\\ q & 1 & 0 & 1 & 0 \\ \end{smallmatrix}) (\begin{smallmatrix}(S_{c(k,l)}  & 1 & 1 & 0 & 0\\ T_{c(k,l)} & 1 & 1 & 1 & 0\\ \end{smallmatrix})\Leftrightarrow  (\begin{smallmatrix}S   & 11 & 10 & 01 & 00\\ \Downarrow \ T & 1 & 1 & 1 & 1\\ \end{smallmatrix}) \end{equation} This can be said by the following: $p \vdash q$ iff whenever $p$ is true, so must $q$. This means  $\lnot(p=1 \land q=0)$, so $p \Rightarrow q$ is never false, so $p \Rightarrow q$ is a tautology. \end{proof} \begin{example} We want to show $(p\vdash p) \Leftrightarrow (p\Rightarrow p)$. We have $\lnot(p=1 \land q=0)=\lnot (1\land 0)=0\lor 1 \mapsto (0+1)(1+1)+1=1$ for left side. $$1 \Leftrightarrow (p\Rightarrow p)\mapsto g(p,q,1):p+q^S+1=1+p(p+1)+1+1=1.$$\end{example} \begin{example} And here $$ (p \vdash .p \lor q) \Leftrightarrow (p\Rightarrow .p \lor q).$$ We use $g(p',q',0)$ for $(p \vdash .p \lor q)$ and  $\lnot(p=0 \land .p\lor q =1)$ gives $((p+1)(pq+1)+1)(0,1)=((0+1)(01+1)+1)=0.$ For the equivalens in $$g(p',q',0): p^S+q^S=0+(p+1)pq=0.$$ 
This proves the statement. We note that even if we do not have any fixed polynomial $g(\vdash)$, it was possible to construct one.\end{example} \begin{example}Next we use a special form $g(p,q',1)$ for $(p\land q) \vdash q \Leftrightarrow (.p\land q \Rightarrow  q)$.
Left side: $(p\land q) \vdash q$ means $\lnot(p\land q)=1 \land q=0)$ is out. We get $(p+1)^Sq^S=((p+1)q+1)q$ and $p+q=p+(p+1)q(q+1)+1=p+(p+1)=1.$
Here, we use a $g$-family specifically for this problem, showing that polynomial forms are very flexible.\end{example} \begin{example}Next one has many conditionals and we try $g(p,q',0)$ here $$(p\Rightarrow q\Rightarrow p)\vdash p \Leftrightarrow (p\Rightarrow q\Rightarrow p) \Rightarrow p.$$ This is mapped into $g(p,q',0):q^Sp=(p+1)(q+1)pqq=0$, which is an easy calculation when we try to get contradictions of the type $p(p+1)$ or $q(q+1)$ which are zero. \end{example} \begin{lemma} We now prove the distribute law $\in S_{c(5,3)}$ using $g(p,q,1)$ $$p\land(q \lor r)= (p\land q) \lor (p \land r)$$ \end{lemma} \begin{proof} We start on the left side, which is mapped to $pq^S$ $$p\land(q \lor r) \mapsto p.(q+1)(r+1)+1=(pq+p)(r+1)+p=pqr+pq+pr=p(qr+q+r)$$ Right side is mapped to $(p^S+1)(q^S+1)+1$ $$(pq+1)(pr+1)+1=pqr+pq+pr=p(qr+q+r)$$ \end{proof}  \begin{example} It is well known that the conditional does not give the converse, usually this is proved by giving a counterexample, which is enough to disprove a theorem, but perhaps it is not satisfactory. Here the $\cdot$ indicates the principal operator \begin{equation} (p\Rightarrow q) \dot{ \Rightarrow } (q\Rightarrow p) \end{equation} We map this into $$g(p,q,1): p^S (q^S +1)+1=p^Sq^S +1+ p^S.$$
The first multiplication and the complement 1 give $$p^Sq^S+1=(p.q+1+1)(.p+1 \ q+1)+1=p+q$$ and this is the XOR. We add and pull back $$p^S: p+q+p(q+1)+1=(p+1)q+1 \hookrightarrow q \Rightarrow  p$$ and find \begin{equation} (p\Rightarrow q) \dot{ \Rightarrow } (q\Rightarrow p)\vdash (q \Rightarrow  p). \end{equation} We see that, in fact, the converse is idempotent for the conditional and vice versa. We show this schematically, and recall that the converse and the conditional vote identically over the inputs $$(\Leftarrow )^{\Rightarrow} = \ \Leftarrow \ (\begin{smallmatrix}\Rightarrow & 1 & 0 & 1 & 1 \\ \Leftarrow & 1 & 1 & 0 & 1 \\ \end{smallmatrix})\mapsto (\begin{smallmatrix}\Leftarrow  & 1 & 1 & 0 & 1 \end{smallmatrix})$$ $$(\Rightarrow)^{ \Leftarrow} = \ \Rightarrow \ (\begin{smallmatrix}\Leftarrow & 1 & 1 & 0 & 1 \\ \Rightarrow & 1 & 0 & 1 & 1 \\ \end{smallmatrix})\mapsto (\begin{smallmatrix}\Rightarrow  & 1 & 0 & 1 & 1 \end{smallmatrix})$$ \end{example} \begin{theorem} We can say that two statements are equivalent if they produce a polynomial that can be reduced to the same normal form. Further, PBNF is a canonical form. \end{theorem} \begin{proof} We note that the polynomials are algebraic expressions of DBNF as follows, and we get the PBNF to the right. $$x\land y  \mapsto pq$$ $$x\land \lnot y \mapsto p(q+1)$$ $$ \lnot x\land y \mapsto (p+1)q$$ $$\lnot x\land \lnot y \mapsto (p+1)(q+1).$$
Every form $S_{c(k,l)}$ is mapped to these particles by a bijection $g$ (which is clear by inspection of the list of transformations), so that $g^{-1}$ exists, and we have $g^{-1}((g(S_{c(k,l)}))=S_{c(k,l)}.$
We cannot have $g(S_{c(k,l)})=g(S^*_{c(k,l)})$ so that two forms is mapped on the same polynomial by DBNF, and $g^{-1}(g(S_{c(k,l)}))=g^{-1}((g(S^*_{c(k,l)}))$ give us $S_{c(k,l)}=S^*_{c(k,l)}$ which is a contradiction. It is also clear that every reduction in the polynomial algebra must end in a minimal polynomial because of the rules of addition and multiplication in $\mathbb{Z}_2$. This means that we, for example, cannot have $g(p \lor \lnot p) \not = g(1) $ so that PBNF is a canonical form. \end{proof} \begin{proposition} We claim that $p\land q \Rightarrow r \Leftrightarrow p \Rightarrow (q \Rightarrow r),$ but their converses are not equivalent $r \Rightarrow  p\land q \not\Leftrightarrow (q \Rightarrow r) \Rightarrow p$ \end{proposition} \begin{proof} $$p\land q \Rightarrow r \mapsto_{g(p,q,1)}  (p^S(q^S+1)+1=[p=pq;q=r]=pq(r+1)+1$$ $$p \Rightarrow (q \Rightarrow r)\mapsto p(q^S+1)+1=[q=(q(r+1)+1]=p(q(r+1)+1+1)+1$$ $$ =pq(r+1)+1.$$ We now use $$p \Rightarrow q \mapsto p^S + q^S+1=pq(r+1)+1 +pq(r+1)+1 +1=1. $$ We find for the converse  $$r \Rightarrow  p\land q \mapsto_{g(p,q,1)} p^S(q^S+1)+1=r(pq+1)+1$$ and for the right side $$(q \Rightarrow r) \Rightarrow p \mapsto  (p+1)q^S+1=(p+1)(q(r+1)+1)+1=(p+1)q(r+1)+(p+1)+1$$ $$=(pq+q)(r+1)+p=pqr+pq+qr+q+p=pq(r+1)+p+q(r+1)$$ $$=(r+1)(pq+q)+p=(r+1)q(p+1)+p.$$ \end{proof} \begin{remark} Here, we use three statements $(p, q, r)$, and we can put the combined statements above in PBNF, but we cannot determine the pull back ($\hookrightarrow$) that returns to the operator space, so we lose the bijection. With just $p$ and $q$, there are 16 polynomials all representing a binary operator. If we just add one more, we get 256 polynomials, each representing an “operator”. We can see this, as we now get 8 particles (atoms) from the letters ($pqr$).
$$x\land y \land z  \mapsto pqr \to 1.$$ 
$$x\land y \land \lnot z \mapsto pq(r+1) \to  2.$$
$$x\land  \lnot y \land z \mapsto p(q+1)r \to 3.$$
$$\lnot x\land y \land z \mapsto (p+1)qr \to 4.$$
$$x\land \lnot y \land \lnot z  \mapsto p(q+1)(r+1)\to 5$$
$$\lnot x\land y \land \lnot z \mapsto (p+1)q(r+1)\to 6$$
$$ \lnot  x\land \lnot y \land z \mapsto (p+1)(q+1)r \to 7$$
$$\lnot x\land \lnot y \land \lnot z \mapsto (p+1)(q+1)(r+1)\to 8$$
Still, the PBNF-transform will determine the truth value of the compound statement using the same method as a truth table.\end{remark} \begin{example} We can find some ternary operators, so for a "weak" bijection, for example in Church {12} again $[p,q,r]\in \mathcal{OP}_*$ there called \emph{conditioned disjunction} which dual to $[r,q,p]\in \mathcal{OP}_{**}$. We know that a generator to binary operators is $(p+1)(q+1)=pq+p+q+1$. To get the full ternary polynomial, we multiply by (r+1): $pqr+pq+pr+qr+p+q+r+1= (p+1)(q+1)(r+1)$. The numberstring in [3] for $[p,q,r]=11100010$ which means that we can construct the polynomial by numbering the areas in the unit square from 1 to 8 as above. This gives \begin{align}[p,q,r] \mapsto pqr+ pq(r+1)+ p(q+1)r +(p+1)(q+1)r=pq+qr+r; \\ N([p,q,r])=11100010 \end{align} The idea with this ternary operator is that it is self-dual and is complete with two constants $([p,q,r],t,f)$. We shall later comment on this fact. \end{example} \begin{example} Other ternary connectives for the pull back are the minority and majority operator which are complement of each other, often written $(Mpqr)$ which gives 1 if at least two are 1. It has $$(p \land q) \lor (p \land r) \dot\lor (q \land r)\mapsto (p^S+1)(q^S+1)+1$$$$=((pq+1)(pr+1)+1)+1)(qr+1)+1$$ $$=(pqr+pq+pr+1)(qr+1)+1=pq+pr+qr.$$ We get the equation $$pq+pr+qr=1 \Leftrightarrow \ 111,110,101,011.$$ so here we could use the PBNF-transformations, and the minority ternary is just to shift reading to $T=0$.\end{example} \begin{example} Still another pullback we find in $+pqr$ the \emph{ternary parity} connective; this will return 1 when the number of inputs of 1 is odd. For a four-place string the inputs are 10 and 01, the operator XOR or | here, but elsewhere often $+$. The adequate inputs are 111,100,010,001 and that give the areas 1,5,6,7 so we get $$pqr+p(q+1)(r+1)+(p+1)q(r+1)+(p+1)(q+1)r=p+q+r. $$ We thus observe that we have the same type of mapping here as for the binary |, $(p+q)$, and it would be a god guess to pick $p+q+r$, because of the Boolean property, $x+x=0$. \end{example} We continue with more examples with compound statements of the form $S_{c(k,3)}.$ \begin{example} We prove that $$(p\land q)\land r =p\land q \land(q\land r).$$ This is mapped to $(pq)r=p(qr)$, and we can use the associativity of real numbers. \end{example} \begin{example} We also prove transitivity under the conditional $$(p\Rightarrow q) \land (q\Rightarrow r)\dot\Rightarrow  (p \Rightarrow r)$$ This is mapped to $p^S(q^S+1)+1$ and $p^S=(p(q+1)+1)(q(r+1)+1)=p(q+1)+q(r+1)+1$ $$q^S=[p(r+1)+1)]$$ $$p^S(q^S+1)+1=p^S([p(r+1)+1)]+1)+1$$ $$=p^Sp(r+1)+1$$ $$[p(q+1)+q(r+1)+1]p(r+1)+1$$ $$=p(q+1)(r+1)+pq(r+1)+p(r+1)+1$$ $$=p(q+1)(r+1)+pq(r+1)+p(r+1)+1$$ $$=pqr+pq+pr+p+pqr+pq+pr+p+1=1$$ \end{example} \section{The Extended Calculus of Statements with Linear and Quadratic Polynomials} We shall here consider a 2-valued function: the value of a singular statement or a compound statement variable is in $(y_1,y_2,\ldots,y_n)$, and the value of an operator is in $(x_1,x_2,\ldots,x_n) $. \begin{definition} $f: \mathcal{D}^{n} \longrightarrow (\mathbb{Z}_n, \mathbb{Z}_n$) is a function of $n$ arguments (operators) in the domain $\mathcal{D}$, null-class operators (0-operators) that can take one of the values ${y_0,y_1,\ldots, y_n} \in \mathbb{Z}n $ or for (1,2)-operators take a number string, a \emph{vector}, of the form $(x_1,x_2,\ldots,x_n)$, often just $(x_1x_2\ldots x_n)$ or even $x_1x_2\ldots x_n$ with $x_n\in \mathbb{Z}n$ which we call the range. \end{definition} We gave the definition here for $n\in \mathbb{N}$ but usually we just consider two truth values $y_0=0$ and $1=y_1$ so that ${y_0,y_1} \in \mathbb{Z}_2 $ and also $x \in \mathbb{Z}_2$ and $|x_n|=1,2,3,4.$ However far riching generalizations exist [2] even into the continium. \begin{example} If we consider the disjunction we get $$f(S_{c(1,0)})(1,0)=f((\lor)(1,0))=(x_1,x_2,x_3,x_4)(1,0)=1110$$ so the argument $\lor$ is sent to a four-placed string of $x$, and then evaluated for here $(1,0)$, but ifjust show the number string we use as befor $N(\lor)=(x_1,x_2,x_3,x_4)(0,1)=(0001)$. We also evaluate statements, for example $$f(p \land (p\Rightarrow q))\mapsto 1/0.$$ Here the arguments are $(p,q, \land,\Rightarrow)\in D.$ Usually, the outcome is decided by logical rules, axioms, or truth tables, but here, instead, we make a transformation to spaces of polynomials $g(p,q,a)$ $$g(f(S_{c(k,l)})) \mapsto (g(p,q,a)\mapsto 1,0 \quad \text{or} \quad g(p,q,1)$$ In this case, we have $$g(f(p \land (p\Rightarrow q)))\mapsto pq^S=p(p(q+1)+1)=pq \mapsto p \land q.$$ Thus we get $f(p \land (p\Rightarrow q))=p \land q $ It is also possible to just count $$f(p \land (p\Rightarrow q))=(1100)(1010)^S=(1100)(1011)=1000$$ here using componentwise multiplication after substitution for the conditional(1011), and this is the same as using polynomials. Later, we will see more of this at work. We shall also write $g(S_{c(k,l)})$ instead of $g(f(S_{c(k,l)})).$ \end{example} Next, we define the notions we will be using in our statement calculus, the geometric view with partitioning the unit. \begin{definition} The unity is divided into $k=2^j$ areas when $j$ is the number of variables. With one variable$(j=1)$, say $x$, there is a partition of the unity in two areas $(k=2)$: $x$ and (not $x$) $\neg x$. The power set $\mathcal {P}(X_1)$ contains $2^k$ members so $\mathcal {P}= {0, x, \neg x, 1}$ and $|\mathcal {P}|=4.$ We have $x,\lnot x$ that partion the unity $1=(x\land \lnot x)$. 
With two variables $j=1,2$ ($x$ and $y$), we get $2^2=4$ areas that partition the unity: $ x \land y, x \land \neg y, \neg x \land y$ and $ \neg x \land \neg y.$ The power set is of size $|\mathcal(P)|=16.$ \end{definition} \begin{definition} There are $2^{2^j}$ members of the type $op_{\boldsymbol{x}}$ in the operator space $\mathcal{OP}$ that represent each an operator connected to a statement letter or between statement letters. Here we consider $\emph{singular}$ operators $(j=1)$, total 4 and $\emph{binary}$ operators $(j=2)$ in a total of 16. \end{definition} Now we want to establish a connection between a $\mathcal{OP}$-space and a polynomial family by the PBNF-transform.\footnote{This type of thinking is used extensively in mathematics, for example, in geometry, a circle with center in the origin and radius 1 is mapped into the algebraic equation $x^2+y^2=1$. Differential operators are also mapped into algebraic forms, when $D_xD_y$ becomes the simple polynomial $\xi\eta$. The idea is that we transfer to a space that is easier to handle, in this case, Boolean polynomial algebra.} \begin{definition} The polynomial House, $\mathcal{H},$ contains polynomial forms, $\emph{symbols}$, where each symbol is connected to a member in the operator space $\mathcal{OP}$ by its range $(x_1,x_2,\ldots,x_n).$ \end{definition} \begin{definition} The polynomial families will all be types of Boolean algebra which is \begin{equation}\mathcal{B} = <B,\land,\lor,',0,1 > ; \land,\lor,' \in op_{\mathcal{B}} \end{equation} where $\land$ and $\lor$ are binary operators in $\mathcal{B}$, $'$ is a singular operator, $0,1$ are designated distinct elements, so $B \in S_c$ is a non-void set upon which these operators act. The binary operators commute, and the distributive laws hold, but not the associative law, so for example $(x \land y) \lor x \not = x \land (y \lor x) $ in general unless $x=y.$  \end{definition} \begin{definition} The symbols we consider here with $j=1,2$ are quadratic polynomials of $p$ and $q$ that form a Boolean ring, using addition and multiplication, modulo 2, where these polynomials are a finite sum of the form $a_1pq+a_2p+a_3q+a_4$ and $a_n=0,1$. $p$ and $q$ in the polynomial have both a geometric meaning as areas dividing the unity, defining a union and a non-void intersection, as well as areas different from $p$ and $q$, which also have a number theoretic meaning representing a number string $p=1100$ and $q=1010.$ We also consider the simpler system with just linear polynomials in $p$ or $q$ for the singular operators, forming a Boolean group with only addition defined. These polynomials are also a finite sum of the form $a_1pq+a_2p+a_3q+a_4$ and $a_n=0,1$ but here $a_1=a_{2|3}=0.$ \end{definition} In the next definition, we consider an extension of the Boolean ring by adding the matrix operation to addition and multiplication. We shall later use this concept for a slightly different view. This is not the case of matrices with elements in a Boolean ring, but instead matrices in a Boolean ring and also elements in this ring. \begin{definition} The Boolean ring can be extended by adding matrix multiplication and addition. We then consider $\left[{x_1 x_2} \atop {x_3 x_4}  \right]$, which uses addition and multiplication, modulo 2. We also use the transpose which can be seen as changing the number string, for example, $p^T(1100)=q(1010)$ or $\left[{x_1 x_3} \atop {x_2 x_4} \right]$.\end{definition} \begin{proposition} For a $2\times 2$-matrix to be Boolean, with Boolean entries, we must have $$\left[{x_1 x_2} \atop {x_3 x_4}  \right] +\left[{x_1 x_2} \atop {x_3 x_4}  \right]=0$$ and $$\left[{x_1 x_2} \atop {x_3 x_4}  \right] \cdot \left[{x_1 x_2} \atop {x_3 x_4} \right]=\left[{x_1 x_2} \atop {x_3 x_4}  \right]$$ and we find that the following matrix-operators $p, p', q, q', I, \land, \uparrow, \iota_0 $ are all Boolean, and they are all closed under addition and multiplication from left $\lor$ multiplication from the the right. \end{proposition} \begin{proof} After testing, we find that the selectors $p, p', q, q', I, \land, \uparrow, \iota_0$ fulfill these criteria as matrices and can be added to the Boolean ring. The XOR(|) does not pass the test, giving $(|)^2=I$ as this is a mirror map of I. The first trivial operator $ \left[{0  0} \atop {0  0} \right]$ is Boolean, but not the second  $\left[ {1  1} \atop {1  1} \right]$. Of the normal form operators, $\land$ is Boolean as a matrix, but not $\lor$. Of the derived operators, $\uparrow$ is Boolean $\downarrow$ is not, and none of the others are either. All matrix-operator elements are also closed under transposition, addition, and multiplication from left $\lor$ right. This is not a restriction, as we often just need one sided multiplication. \end{proof} We already here show an example of this where we use matrices as generators. \begin{theorem} The matrix-definition makes it possible to start with a Selector as a mono-connective primitive for a formulation of Propositional Calculus. \end{theorem} \begin{proof} We start with four matrix operations and here we use $p$ as Selector. 
\begin{equation}\left\{  
\begin{array}{ll}
\boldsymbol{x}_{p}\mapsto (\begin{smallmatrix}1 & 1\\0 & 0\\ \end{smallmatrix})\\
\boldsymbol{x}_{q=p^T}\mapsto (\begin{smallmatrix}1 & 0\\1 & 0\\ \end{smallmatrix})\\
\boldsymbol{x}_{p'}\mapsto (\begin{smallmatrix}0 & 0\\1 & 1\\ \end{smallmatrix})\\
\boldsymbol{x}_{p\times q}\mapsto (\begin{smallmatrix}0 & 0\\0 & 0\\ \end{smallmatrix})
\end{array}\right.
\end{equation} Now we have the usual Selectors $p, q$ and $p'$. We may add, to get 1 (which is not a Boolean matrix)  $p+p'=p+p+1 =1$ witch are defined because both $p$, $p'$ and 0 exist as Boolean matrices and by that they can also be mapped to the polynomials. Recall that we do not demand to create the trivial operators by our primitives, it is enough that these can be defined later, by Church, page 129, § 24 in [3].
After this start with $p$, we have the full polynomial, by adding and multiplying our defined components, to get $pq+p+q+1 \in \mathcal{H}(g(p,q,1)),$ and can construct any operator by the same method as we did before.  It was enough with $p, q, 0$, and $1$, which shows the importance of the Selectors and the Trivial operators.\end{proof}
\section {Linearity and proof of the first Polynomial Normal Form Theorem} A boolean polynomial is linear for $p,q \in S_{c(k,l)}$ (the statement form, $c$=classical, $k$=counting the operators(conectives), $l$=total of litterals) the following iterated sence. \begin{definition} We use $\star$ as indicating some space or some operator \begin{align} g(S_{c(k,l)}\star S_{c(m,n)})=g(g(S_{c(k,l)}) \star g(S_{c(m,n)}))=g(p,q,a)g(\star)g(p,q,a)\\ =(a_1pq+a_2p+a_3q+a)(0,1)=(0|1) \big{|}g(p,q,a); \ a \land a_{j=1,2,3}=0/1. \end{align} \end{definition} \begin{example}Let the conjunctive compound statement $S_{c(2,3)}$ with two conditional sub statements  $S_{c(1,2)}^1$ and  $S_{c(1,2)}^2$  be $(p\Rightarrow q)\land (q\Rightarrow p).$ Then we have for  $$g_{(p,q,1)}((p\Rightarrow q)\land (q\Rightarrow p))= g(g(p\Rightarrow q) \land g (q\Rightarrow p))$$ $$(p(q+1)+1) g(\land) (p+1)q+1) = (p(q+1)+1)\cdot((p+1)q+1)$$ $$=(p(q+1)+1)+((p+1)q+1)=p+q+1(1,0)=1(00,11)|0(10,01).$$ Here we evaluated for $1 \land 0$, and, as we are in $g(p,q,1)$, this gives the number string (1001). The inverse function $g^{-1}(p,q,1)$ is applied: $$g^{-1}_{(p,q,1)}(p+q+1)= p\Leftrightarrow q.$$ We also see that the range $\bar{\mathcal{H}}$ is real-valued, we think of the polynomial as a number, not in the sense that is often used in algebra where there is no evaluation, is called indeterminate, but that view does not apply here. \end{example}\begin{remark}This iteration process could continue until all sub-statements are finished and the main connective is transformed. We refer to this process as not having any rules of inference in PBNF. Usually, we apply the substitution, which is traditionally regarded as a rule of inference. The iterative process perhaps clarifies better how our formulation of PC works. We also argue that this is not a logical rule of inference, but linearity used in mathematical areas, as for example in $f(a+b)=f(a)+f(b)$.\end{remark} The mapping $f$ in the singular case $(j=1)$ is defined by \begin{equation} f(x_1,x_2)= (0, p, p+1,1) \end{equation} in a Boolean group $\mathcal{B}_1= \langle 0, p, p+1,1;+ \rangle.$ In the binary case, we have a Boolean polynomial matrix-ring \begin{align} \mathcal{B}_2=\langle 0,\left[{0 0} \atop {0 0} \right]p, \left[{1 1} \atop {0 0}  \right]q, \left[{1 0} \atop {1 0}  \right]p+1,\left[{0 0} \atop {1 1}  \right] q+1,\left[{0 1} \atop {0 1}  \right] p+q, p+q+1,\left[{1 0} \atop {0 1}  \right]\\  (p+1)q,(p+1)q+1, p(q+1), p(q+1)+1,\\  (p+1)(q+1)\left[{ 0} \atop {0 1}  \right],(p+1)(q+1)+1, pq,\left[{1 0} \atop {0 0}  \right] pq+1, 1;+, \cdot,\rangle. \end{align} The mapping $g^{-1}:\mathcal{H} \longrightarrow \mathcal{OP} $  is the $\emph{fiber}$ or pull back, defined by the inverse image of $g(\boldsymbol {x}).$ Observe, that we just demand a one-sided multiplication for the matrix-multplication for closedness. \begin{remark} Here we must be carefull as the singular operators, $(=,-,\lnot,+$, and we will explore them in next article) could also be defined $-:= p \land \lnot p$ and $+:=  p \lor \neg p$ and then we have no bijection as $$g^{-1}(p(p+1))\mapsto -(00)\quad \text{and} \quad \iota_0 (0000)$$ and the same for $\iota_1$. We shall avoid this by not using these binary polynomials for the singular operators. \end{remark} We let $N$ denote the number string (without commas) of an operator, so, for example, $N(p\Rightarrow q)\sim N(\Rightarrow) = (1011)$. The connection is $g(1011)=p(q+1)+1$ and the fiber $g^{-1} (p(q+1)+1)=op_{\boldsymbol{x}(1011)}.$ 
\begin{theorem} \begin{tabular}{l l l}
$N(p .op_{\boldsymbol{x}} q)$ & String in $f(op_{\boldsymbol{x}})$ & Symbol $\in \mathcal{H}^*  $\\
\hline
$(1) N(p\wedge q)=1000$ &$(1100)(1010)=(1000)$& $pq$ \\
$(2_p)N(p,\neg p) $&$(1100),(0011)$ &$p,p+1$\\
$(2_q)N(q,\neg q) $&$(1010),(0101)$ &$q,q+1$\\
$(3) N(p\vee q)=1110$&$1100+1010+1000=1110$ &$(p+1)(q+1)+1$ \\
$(4) N(p\Rightarrow q)=1011$ &$1000+1100+1111=1011$&$p(q+1)+1$\\
$(5) N(p\Leftarrow q)=1101$ &$1000+1010+1111=1101$&$(p+1)q+1$\\
$(6) N(p\Leftrightarrow q) =1001$ &$(1011)(1101)=1001$ &$p+q+1$ \\
$(7) N(p\uparrow q)= 0111$ &$(1100)(1010)+1111=0111$ &$pq+1$. \\
$(8) N(p\downarrow q)=0001$ & $1110+1111=0001$ &$(p+1)(q+1)$. \\
$(9) N(p|q)=0110$ & $1100+1010=0110$ &$p+q$. \\
\end{tabular} \end{theorem} 
\begin{proof}To find these polynomials, we start with $pq$, which we name the \emph{identity polynomial} and we use it as the \emph{primitive} polynomial and we assume that this maps to $(1,0,0,0)$, we build on that operator, and test the result. We multiply component wise(inner product) $$\langle p\cdot q\rangle=\langle 1100 \cdot 1010\rangle=(1,0,0,0)\mapsto \land.$$  Observe that this is not a scalar product; the result is a vector. Now we try to get the disjunction $ p \lor q$ by adding $p+q=(1,1,0,0)+(1,0,1,0)=(0,1,1,0) \mapsto | $, which gives the exclusive disjunction, since $1+1=0$ so the intersection is missing. This means that the intersection $$\cap =pq\mapsto \land $$ is multiplication, but the union $$\cup \not = p+q \mapsto |, \quad \cup = p+q+pq \mapsto \lor $$ is not just addition of $p$ and $q$, but by adding $(| + \land)\mapsto p+q+pq=(1,1,1,0)$ using the primitive we started with. We can also see this as the definition of the exclusive disjunction |. \begin{equation}p+q=(p\cap q')\cup (p'\cap q). \end{equation} We often write $pq+p+q$  as $(p+1)(q+1)+1 \hookrightarrow \lor $ because $\lor$ is the complement of the homogenous input $ \left[ {0 \atop 0} \right] $ in $ \left[ {1100 \atop 1010} \right]$ as the complement is done by adding $1, ((0001)+ (1111)=(1110)).$ We like to use  these \emph{particles}  $(p+1)$ and $(q+1)$ in the calculations, because both are $0$ if they are multiplied by the same Selector as in the parentheses ($p(p+1)=0)$.
In this way, we proved the statement, as all these transformations could be tested one by one. We notice the use of the building blocks $(0, p, q, p+1, q+1,1)$ that use only + and $\cdot$ in $\mathbb{Z}_2$ to fill the table. By adding 1, you negate the operator $ +1 \leftrightarrow \neg op_x$, by adding $p, q$ and $1$ you add component wise $p+q+1=(1,1,0,0) +(1,0,1,0)+(1,1,1,1)=(0,1,1,0)+(1,1,1,1)=(1,0,0,1)$ with $1+1=0$ which gives $ \Leftrightarrow, $ or you can use multiplication as above. 
Here we notice that all forms are in their complement as we have added 1 to all but the first. This points to the possibility of a dual space that is not in the complement. This case we consider in the next article. To get the polynomials for the trivial operators N(1)=1111 and N(0)=0000, we can take any operator and add what is missing. \end{proof} \begin{example} We start with $1110$ and add $0001$ to get $1111$ and with polynomials $$(p+1)(q+1)+1 + (p+1)(q+1)=1.$$ We also notice that if we add $$N(op{\boldsymbol{x}}+op_{\boldsymbol{x}})=N(op_{\boldsymbol{x}})+N(op_{\boldsymbol{x}})=0$$ so we see that  $$N: op_{\boldsymbol{x}}\longrightarrow x_1x_2x_3x_4; \quad x_n=0,1$$ is a linear operator, defined by giving the operator number string, and we get the additive neutral element as result. Further more $$op_{\boldsymbol{x}}^2=op_{\boldsymbol{x}}$$ We see the idempotens. So this is all Boolean on the operator level.\end{example} 
In the next article, we will discuss, among other things, primitive operators and polynomial generators. Some other items in next article are: We complete our formulation of PC, by adding to our Normal Polynomial Family the three other families, which are useful for our tecnique.
For linear polynomials, we show that they form the full symmetry of the 4-Klein group and that it is possible to define the canonical decomposition of the negation [5]. For the quadratic polynomials, we show that the equality $(p\uparrow p)=\lnot p$ is undecidable in the operatorspace, but decidable in PBNF. We continue to look into the primitive operators and generators for the polynomial families.

\end{document}